\documentclass[12pt]{article}
\usepackage{amsmath,amssymb,mathrsfs,amsthm}

\newtheorem{thm}{Theorem}
\newtheorem{prop}{Proposition}
\newtheorem{corollary}{Corollary}
\newtheorem{lemma}{Lemma}

\theoremstyle{definition}

\newtheorem{example}{Example}

\theoremstyle{remark}
\newtheorem{remark}{Remark}

\begin{document}

\title{Remarks on Semiclassical Wavefront Set}
\author{Kentaro Kameoka}

\date{}

\maketitle
\begin{abstract}
The essential support of the symbol of a semiclassical 
pseudodifferential operator is characterized by semiclassical wavefront sets of distributions.
The proof employs a coherent state whose center in phase space is dependent on Planck's constant.

\end{abstract}

\section{Introduction}
   In microlocal analysis, it is well known that the wavefront set $\mathrm{WF}(A)$ of a pseudodifferential operator
$A$ is characterized as follows;
\[
\mathrm{WF}(A)=\bigcup _u \mathrm{WF}(Au)
\]
where $u$ ranges over all distributions with compact support (H\"ormander~\cite[Corollary 28.1.26]{H}). 
Although an extension to the semiclassical setting is stated in Zworski~\cite[Theorem 8.1.6]{Z}, (which is our Corollary 1 with $\delta=0$), 
the proof seems to require a more precise argument.
The extension to the semiclassical setting claims that the semiclassical wavefront set $\mathrm{WF}_h(A)$ of a semiclassical pseudodifferential 
operator $A$ is characterized as follows;
\[
\mathrm{WF}_h(A)=\bigcup _u \mathrm{WF}_h(Au)
\]
where u ranges over all families $u=\{ u(h) \}_{0<h<h_0} \subseteq L^2 (\mathbb{R}^n)$ which are bounded with respect 
to $h$.
While the proof in \cite{Z} of the inclusion 
\[
\mathrm{WF}_h(A)\subseteq\bigcup _u \mathrm{WF}_h(Au)
\] 
simply indicates the use of the coherent state, 
it is possible that $(x_0, \xi_0)\not\in \mathrm{WF}_h(A\phi _{x_0,\xi_0 ,h})$ even if $(x_0, \xi_0)\in \mathrm{WF}_h(A)$,
where $\phi _{x_0,\xi_0 ,h}$ is the coherent state centered at $(x_0,\xi_0)$ (see Section 3).
This shows the difficulty of the problem and the proof requires a more detailed argument than that in \cite{Z}. 
In this note, we give a complete proof of this result using the coherent state whose center is dependent on the semiclassical parameter. 

For any $0\not \equiv \phi \in \mathscr{S}(\mathbb{R}^n)$ (the space of Schwartz functions), we define the corresponding coherent state
$\phi _{x_0,\xi_0 ,h}\in \mathscr{S}(\mathbb{R}^n)$ centered at $(x_0,\xi_0)\in T^*\mathbb{R}^n$ 
by
\[  \phi _{x_0,\xi_0 ,h}(x)=h^{-\frac{n}{4}} \phi \left(\frac{x-x_0}{h^{\frac{1}{2}}}\right)
  e^{i\langle x, \xi_0 \rangle /h}.\]
  
 We next recall the basic properties of the semiclassical wavefront set following Zworski~\cite[Chapter 8]{Z} 
(see also \cite{A}, \cite{CR}, \cite{G}, \cite{G1}, \cite{G2}, \cite{M}). 
The symbol class $S_{\delta}$ with $\delta \ge 0$ is defined as
\[ 
S_{\delta}=\{ a(x,\xi;h)\in C^{\infty}(T^*\mathbb{R}^n) | \mspace{6mu} |\partial _x ^\alpha \partial _\xi ^\beta
 a(x,\xi; h)|\le C_{\alpha ,\beta}h^{-\delta (|\alpha|+|\beta|)} \}.
\]
 We set $S=S_0$. A semiclassical pseudodifferential operator with symbol $a\in S_{\delta}$ is defined as
\[
Op_{t,h}(a)u(x)=(2\pi h)^{-n} \iint a(tx+(1-t)y,\xi;h)e^{i\langle x-y, \xi \rangle /h} u(y) dy d\xi 
\]
where $t\in [0,1]$ is the quantization parameter. We note that 
\begin{align*}
&Op_{1,h}(a)u(x)=\mathcal{F}_h^{-1}(a(x,\xi;h)\mathcal{F}_h u(\xi))(x), \\
&Op_{0,h}(a)u(x)=\mathcal{F}_h^{-1}(\mathcal{F}_h (a(y,\xi;h)u(y))(\xi))(x),
\end{align*}
where $\mathcal{F}_h$ is the semiclassical Fourier transform with the convention 
\[
\mathcal{F}_h u(\xi)=(2\pi h)^{-\frac{n}{2}}\int u(x)e^{-i\langle x,\xi \rangle /h} dx.
\]

For $a \in S_{\delta}$ with $0\le \delta < \frac{1}{2}$, we say that $(x_0, \xi _0)\not\in \mathrm{ess}\text{-}\mathrm{supp}(a) 
\subseteq T^*\mathbb{R}^n$ if and only if there exists an open neighborhood $U$ of $(x_0, \xi _0)$ such that
$a=\mathcal{O}_{C^{\infty}(U)}(h^{\infty})$.
The semiclassical wavefront set $\mathrm{WF}_h(A)$ of a semiclassical pseudodifferential operator $A=Op_{t,h}(a)$ is 
defined as $\mathrm{WF}_h(A)=\mathrm{ess}\text{-}\mathrm{supp}(a)$.
It is well known that if $Op_{t,h}(a)=Op_{s,h}(b)$, then $b(x,\xi;h) \sim \sum_{j=0}^{\infty}
\frac{1}{j!} \left( i (s-t)h \langle D_x,D_{\xi} \rangle \right)^j a(x,\xi;h)$. Thus   
 $\mathrm{ess}\text{-}\mathrm{supp}(a)=\mathrm{ess}\text{-}\mathrm{supp}(b)$ and
the definition of the semiclassical wavefront set $\mathrm{WF}_h(A)$ is independent of 
the quantization parameter $t\in [0,1]$ if $\delta \in [0,\frac{1}{2})$.

We next recall the definition of the semiclassical wavefront set $\mathrm{WF}_h (u) \subseteq T^*\mathbb{R}^n$ 
of a family of distributions $u=\{u(h)\}_{0<h<h_0}$. For simplicity, we assume that $\|u(h)\|_{L^2}\le C$ for
some constant $C$. 
 We say that $(x_0,\xi_0) \not\in \mathrm{WF}_h (u)$ if and only if there exists a symbol $a\in S$ with
$\inf_h |a(x_0,\xi_0;h)|>0$ such that $Op_{t,h}(a)u= \mathcal{O}_{L^2}(h^{\infty})$. 
Note that this definition is independent of the quantization parameter $t\in [0,1]$ since $Op_{t,h}(a)=Op_{s,h}(b)
$ with $a-b=\mathcal{O}_{S}(h)$.
It is known (see Zworski~\cite[Theorem 8.1.5]{Z}) that 
\[ \mathrm{WF}_h(Au) \subseteq \mathrm{WF}_h (A) \cap \mathrm{WF}_h (u)\]
if the symbol of $A$ is in $S_{\delta}, (0\le \delta < \frac{1}{2})$, which proves the inclusion $\supseteq$ in 
Corollary 1. 

Our main theorem is the following:
\begin{thm}
Suppose that $a\in S_{\delta}, (0\le \delta < \frac{1}{2})$, and $(x_0, \xi _0)\in \mathrm{ess}\text{-}\mathrm{supp}(a)$.
Then there exists $(x_h,\xi_h)\in T^*\mathbb{R}^n$ such that $(x_0, \xi _0)\in \mathrm{WF}_h(Op_{1,h}(a) 
\phi _{x_h,\xi_h ,h})$ for every $\phi \in \mathscr{S}(\mathbb{R}^n)$ whose Fourier transform is nonnegative and has compact support.
\end{thm}
This leads to a characterization of the semiclassical wavefront set of a semiclassical pseudodifferential operator: 
\begin{corollary}
For a semiclassical pseudodifferential operator $A$ whose symbol is in $S_{\delta}, (0\le \delta < \frac{1}{2})$, 
\[
\mathrm{WF}_h(A)=\bigcup _u \mathrm{WF}_h(Au)
\]
where u ranges over all families $u=\{ u(h) \}_{0<h<h_0} \subseteq L^2 (\mathbb{R}^n)$ which are bounded with respect 
to $h$.
\end{corollary}
\begin{remark}
The semiclassical wavefront set $\mathrm{WF}_h$ in this paper coincides with the finite part of the semiclassical wavefront set $\mathrm{WF}_h^f$ in 
\cite{A}, \cite{G}.
\end{remark}
\begin{remark}
Theorem 1 is not true in general if we set $(x_h,\xi_h)=(x_0 ,\xi_0)$ even if the symbol has an asymptotic expansion
of the form $a(x,\xi;h) \sim \sum_{j=0}^{\infty} h^j a_j (x,\xi)$ (see Section 3). In the proof of Theorem 1, we
choose $(x_h,\xi_h)$ carefully.
\end{remark}
In the Appendix, we give a lower bound on $\mathrm{WF}_h(b(x)e^{\frac{i}{h}\phi(x)})$ 
assuming no condition on $b$ or $\phi$ whose proof is close to that of Theorem 1. 

\section{Coherent State}
In this section, we recall the basic properties of the coherent state.
The semiclassical Fourier transform is given by 
\[\mathcal{F}_h \phi _{x_0,\xi_0 ,h}(\xi)
 =h^{-\frac{n}{4}}e^{i\langle x_0, \xi_0 \rangle /h} \hat{\phi} \Bigl(\frac{\xi-\xi_0}{h^{\frac{1}{2}}}\Bigr)
  e^{-i\langle x_0, \xi \rangle /h}.\]
\begin{prop}
Let $a\in S$, $\phi \in \mathscr{S}(\mathbb{R}^n)$ and $r_0>0$. Then 
\[Op_{t, h}(a)\phi _{x_0,\xi_0 ,h}=\mathcal{O}_{L^2}(h^{\infty})\]
uniformly for $(x_0, \xi_0)\in T^*\mathbb{R}^n$ with $\mathrm{dist}((x_0, \xi_0), \mathrm{supp}a)>r_0$.
\end{prop}
\begin{proof}
Since $\mathrm{dist}((x_0, \xi_0), \mathrm{supp}a)>r_0$, we can take $\chi \in C_c^{\infty}(\mathbb{R}^n)$ which equals to $1$ near $0$ and 
$b, c\in S$ such that $a=b+c$ and $\mathrm{supp}\chi(x-x_0)\cap \mathrm{supp}b=\emptyset$, 
$\mathrm{supp}\chi(\xi-\xi_0)\cap \mathrm{supp}c=\emptyset$. 
The standard semiclassical calculus implies 
$Op_{t, h}(b)\chi(x-x_0)=\mathcal{O}_{L^2 \to L^2}(h^{\infty})$ and $Op_{t, h}(c)\chi(hD-\xi_0)=\mathcal{O}_{L^2 \to L^2}(h^{\infty})$. 
We also have $(1-\chi(x-x_0))\phi _{x_0,\xi_0 ,h}=\mathcal{O}_{L^2}(h^{\infty})$ and $(1-\chi(hD-\xi_0))\phi _{x_0,\xi_0 ,h}=\mathcal{O}_{L^2}(h^{\infty})$ 
since $\phi \in \mathscr{S}(\mathbb{R}^n)$ and $\chi=1$ near $0$. These complete the proof since 
\begin{align*}
Op_{t, h}(a)&\phi_{x_0,\xi_0 ,h}=Op_{t, h}(b)\chi(x-x_0)\phi _{x_0,\xi_0 ,h}+
Op_{t, h}(b)(1-\chi(x-x_0))\phi _{x_0,\xi_0 ,h}\\
&+Op_{t, h}(c)\chi(hD-\xi_0)\phi _{x_0,\xi_0 ,h}
+Op_{t, h}(c)(1-\chi(hD-\xi_0))\phi _{x_0,\xi_0 ,h}.
\end{align*}
\end{proof}
\begin{corollary}
Suppose that $(x_0,\xi_0) \not\in \mathrm{WF}_h (u)$, where $u=\mathcal{O}_{L^2}(1)$. 
Then there exists a neighborhood $V$ of $(x_0,\xi_0)$ such that 
\[(u, \psi_{x_h, \xi_h, h})_{L^2}=\mathcal{O}(h^{\infty})\]
for any $\psi \in \mathscr{S}$ and $(x_h, \xi_h)\in V$.
\end{corollary}
 \begin{proof}
By the definition of $\mathrm{WF}_h (u)$, there exists a symbol $a\in S$ with
$\inf_h |a(x_0,\xi_0;h)|>0$ such that $Op_{t,h}(a)u= \mathcal{O}_{L^2}(h^{\infty})$. 
The standard symbol calculus implies that there exists $b\in S$ such that 
\[Op_{t,h}(b)Op_{t,h}(a)=Op_{t,h}(\chi)+\mathcal{O}_{L^2 \to L^2}(h^{\infty}),\] 
where $0\le \chi \in C_c^{\infty}(T^*\mathbb{R}^n)$ is $h$-independent and equals to $1$ near $(x_0, \xi_0)$. 
We note that 
\[(u, \psi_{x_h, \xi_h, h})_{L^2}=(u, Op_{1-t, h}(1-\chi)\psi_{x_h, \xi_h, h})_{L^2}+(Op_{t,h}(\chi)u, \psi_{x_h, \xi_h, h})_{L^2}.\]
Proposition 1 implies 
$Op_{1-t, h}(1-\chi)\psi_{x_h, \xi_h, h}=\mathcal{O}_{L^2}(h^{\infty})$ if $(x_h, \xi_h)$ is sufficiently close to $(x_0, \xi_0)$. 
We also have $Op_{t,h}(\chi)u=\mathcal{O}_{L^2}(h^{\infty})$ since $Op_{t,h}(a)u= \mathcal{O}_{L^2}(h^{\infty})$. 
These complete the proof.
\end{proof}
\begin{remark}
Corollary 2 is often stated when $\psi$ is Gaussian (see \cite[Proposition 22]{CR}) and closely related to the FBI transform (see \cite[Theorem 13.14]{Z}). 
We use Corollary 2 for $\psi \in C_c^{\infty}(\mathbb{R}^n)$.
\end{remark}

\section{Counter Examples}
Using Corollary 2 it is easy to verify that $\mathrm{WF}_h (\phi _{x_0, \xi_0, h})=\{(x_0,\xi_0)\}$ for every $\phi \in \mathscr{S}(\mathbb{R}^n)$. 
However, $(x_0,\xi_0)\in \mathrm{WF}_h (A)$ does not imply $\mathrm{WF}_h (A \phi _{x_0, \xi_0, h})=\{(x_0,\xi_0)\}$, 
which suggests that the proof of Corollary 1 requires a precise argument. Here is an intuitive 
counter example;
\begin{example}
Take $h$-independent 
$a _0 \in S$ such that $(0,0)\in \mathrm{supp}a_0$ and $a_0 (x,\xi)=0 \mspace{6mu} \mathrm{if}\mspace{6mu}
 x_1<0$. If we set $a(x,\xi;h)=a_0 (x_1 -h^{\frac{1}{4}},x_2, \dots ,x_n,\xi)$, it is obvious that $a\in 
 S$ and $(0,0)\in \mathrm{ess}\text{-}\mathrm{supp}(a)$. 
 We have $\|\phi _{0,0,h}\|_{L^2(\{x_1 \ge \frac{1}{2}h^{\frac{1}{4}}\})}=\mathcal{O}(h^{\infty})$ since $\phi \in \mathscr{S}(\mathbb{R}^n)$.
 Thus $\|Op_{t,h}(a)\phi _{0,0,h}\|_{L^2}=\mathcal{O}(h^{\infty})$ and hence
 $(0,0)\not\in \mathrm{WF}_h(Op_{t,h}(a)\phi _{0,0,h})$.
\end{example}
One can also find a counter example $a$ which has an asymptotic expansion
of the form $a(x,\xi;h) \sim \sum_{j=0}^{\infty} h^j a_j (x,\xi)$;
\begin{example}
Take $h$-independent $a_j \in S$ such that $(0,0) \not\in \mathrm{supp}a_j$ and
\[\lim_{j\to \infty}\mathrm{dist}((0,0), \mathrm{supp}a_j)=0.\]
 By Borel's theorem, we can take $a\in S$ such that 
$a(x,\xi;h) \sim \sum_{j=0}^{\infty} h^j a_j (x,\xi)$. 
Then $(0,0)\in \mathrm{ess}\text{-}\mathrm{supp}(a)$. On the other hand, 
$\|Op_{t,h}(a)\phi _{0,0,h}\|_{L^2}=\mathcal{O}(h^{\infty})$ since $\|Op_{t,h}(a_j)\phi _{0,0,h}\|_{L^2}=\mathcal{O}(h^{\infty})$ for 
any $j$ by Proposition 1. Thus 
 $(0,0)\not\in \mathrm{WF}_h(Op_{t,h}(a)\phi _{0,0,h})$. 
\end{example}

\section{Proof of Theorem 1}
\begin{prop}
Suppose $a\in S_{\delta}, \mspace{6mu} \delta \ge 0$ and $\|a \|_{L^{\infty}}=\mathcal{O}(h^{\alpha}),
\mspace{6mu} \alpha \ge 0$. Then, 
$\|h^{\delta} \partial a\|_{L^{\infty}}=\mathcal{O}(h^{\alpha (1-\varepsilon)})$ for any $\varepsilon >0$.
\end{prop}

If we apply Proposition 2 to the derivatives of $a$, we learn 
\begin{corollary}
Under the assumption of Proposition 2, $a=\mathcal{O}_{S_{\delta}}
(h^{\alpha (1-\varepsilon)})$ for any $\varepsilon >0$.
\end{corollary}

\begin{remark}
By considering $a(h^{\delta}x, h^{\delta}\xi)\in S$, Proposition 2 is reduced  to the case of $\delta=0$ , which is a consequence of 
Kolmogorov's general theorem from \cite[Section 6]{S}. 
We give a proof of Proposition 2 for the sake of completeness.
\end{remark}

\begin{proof}[Proof of Proposition 2]
Recall the well known gradient estimate (for example \cite[Lemma 4.31]{Z}), 
$\|\partial f\|_{L^{\infty}} \le C \| f\|_{L^{\infty}}^{\frac{1}{2}} \|\partial^2 f\|_{L^{\infty}}^{\frac{1}{2}}$
for any bounded $f\in C^{\infty}(\mathbb{R}^n)$.  
We assume that Proposition 2 is proved for $\varepsilon=\varepsilon_n$ $(\varepsilon_0 =1)$. If we apply Proposition 2 for 
$\varepsilon=\varepsilon_n$ to $h^{\delta}\partial a \in S_{\delta}$ with 
$\|h^{\delta}\partial a\|_{L^{\infty}}=\mathcal{O}(h^{\alpha (1-\varepsilon_n)})$,
we obtain
$\|h^{2\delta}\partial^2 a\|_{L^{\infty}}=\mathcal{O}(h^{\alpha (1-\varepsilon_n)^2})$.
We apply the gradient estimate to $a$ and obtain 
$\|h^{\delta}\partial a\|_{L^{\infty}}
=\mathcal{O}(h^{\frac{1}{2}\alpha+\frac{1}{2} \alpha(1-\varepsilon_n)^2})$.
Thus Proposition 2 is proved for 
$\varepsilon=\varepsilon_{n+1}$ with $1-\varepsilon_{n+1}=\frac{1}{2}+\frac{1}{2}(1-\varepsilon_n)^2$. 
It is then easy to see that $\lim_{n \to \infty}
\varepsilon_n=0$.
\end{proof}

\begin{proof}[Proof of Theorem 1]
We assume $a\in S_{\delta}, \mspace{6mu} \delta\in [0,\frac{1}{2})$ 
and $(0,0)\in \mathrm{ess}\text{-}\mathrm{supp}(a)$. 
This means by definition that $a\not =\mathcal{O}_{C^{\infty}(B(r))}(h^{\infty})$ for any $r>0$, 
if we denote $ B(r)=\{(x,\xi)\in T^*\mathbb{R}^n|\mspace{6mu} |x|^2+|\xi|^2<r^2\}$. 
In fact, this is equivalent to saying that
$a\not =\mathcal{O}_{L^{\infty}(B(r))}(h^{\infty})$ for any $r>0$. To see this, take $\chi \in C^{\infty}_c (B(r))$ 
such that $\chi= 1$ on $B(\frac{r}{2})$. If $a=\mathcal{O}_{L^{\infty}(B(r))}(h^{\infty})$, 
then $a\chi\in S_{\delta}$ and $\|a\chi\|_{L^{\infty}} =\mathcal{O}(h^{\infty})$. 
The gradient estimate applied to the derivatives of $a\chi$ implies $a\chi\ =\mathcal{O}_{C^{\infty}}(h^{\infty})$.
Thus $a=\mathcal{O}_{C^{\infty}(B(\frac{r}{2}))}(h^{\infty})$, which is a contradiction. 

Set $\alpha (r)=\sup \{\alpha| a=\mathcal{O}_{L^{\infty}(B(r))}(h^{\alpha})\}$. This is finite by the above remark. 
Since $\alpha (r)$ is monotone, we can take $r_1>r_2> \dots \to 0$ such that 
$\alpha (r)$ is continuous at $r=r_j$ for $j=1,2,\dots$. Set $\alpha_j=\alpha(r_j)$. 

In the following, we fix $\varepsilon>0$ so that $\varepsilon<\frac{1}{2}(\frac{1}{2}-\delta)$. 
We first give an upper bound on the first derivatives.
\begin{lemma}
For any $j=1,2,\dots$,
there exists $s>0$ such that 
\[\|h^{\delta} \partial a\|_{L^{\infty}(B(r_j+s))}=\mathcal{O}(h^{\alpha_j -\varepsilon}).\]
\end{lemma}
\begin{proof}
This follows from Proposition 2, the continuity of $\alpha(r)$ at $r=r_j$ and the cut off argument as above. 
\end{proof}

We next choose the center $(x_h,\xi_h)$ of our coherent state. 
\begin{lemma}
There exists $\{(x_h,\xi_h)\} \subseteq T^*\mathbb{R}^n$ such that
for every $j\ge 1$ there exists a sequence $h_1, h_2\dots \to 0$ which satisfies 
$\{(x_{h_k},\xi_{h_k})\}_{k=1}^{\infty}\subseteq B(r_j)$ and 
\[ |a(x_{h_k},\xi_{h_k};h_k)| \ge h_k^{\alpha_j +\varepsilon}, \mspace{10mu} k=1,2,\dots. \]
\end{lemma}
\begin{proof}
We can take $0<h_1<1$ and $(x_1,\xi_1)\in B(r_1)$ such that  
\[ |a(x_1,\xi_1;h_1)| \ge h_1^{\alpha_1 +\varepsilon}\]
by the definition of $\alpha(r)$.
We can take $0<h_2<h_1$ and $(x_2,\xi_2)\in B(r_2)$  such that 
\[ |a(x_2,\xi_2;h_2)| \ge h_2^{\alpha_2 +\varepsilon}.\]
by the same reason.
We next choose similarly $0<h_5<h_4<h_3<\min\{h_2,2^{-1}\}$ 
and $(x_3, \xi_3)\in B(r_1),\mspace{6mu}(x_4, \xi_4)\in B(r_2),
\mspace{6mu}(x_5, \xi_5)\in B(r_3)$ such that
\[|a(x_3,\xi_3;h_3)| \ge h_3^{\alpha_1 +\varepsilon}, \mspace{6mu}|a(x_4,\xi_4;h_4)| \ge h_4^{\alpha_2 +\varepsilon},
\mspace{6mu}|a(x_5,\xi_5;h_5)| \ge h_5^{\alpha_3 +\varepsilon}.\]
We next choose similarly $0<h_9<h_8<h_7<h_6<\min\{h_5,3^{-1}\}$ 
and $(x_6, \xi_6)\in B(r_1),\mspace{6mu}(x_7, \xi_7)\in B(r_2),
\mspace{6mu}(x_8, \xi_8)\in B(r_3), \mspace{6mu}(x_9, \xi_9)\in B(r_4)$ which satisfy similar estimates.
We repeat this process and obtain a sequence $h_1>h_2>\dots \to 0$. 
We set $(x_h,\xi_h)=(x_j,\xi_j)$ if $h=h_j$ and otherwise
$(x_h,\xi_h)=(0,0)$. Then Lemma 2 is easily verified.
\end{proof}

Take $0 \not\equiv\phi\in\mathscr{S}(\mathbb{R}^n)$ 
such that the Fourier transform $\hat{\phi}=\mathcal{F}_1 \phi$ is nonnegative
and of compact support. 
We now prove $(0,0)\in \mathrm{WF}_h(Op_{1,h}(a) \phi _{x_h,\xi_h ,h})$. 
Assume on the contrary that $(0,0)\not\in \mathrm{WF}_h(Op_{1,h}(a) \phi _{x_h,\xi_h ,h})$.
We take $j \ge 1$ such that $B(r_j)\subset V$ ($V$ is that in Corollary 2).
By Lemma 2, we can take a sequence $h_1, h_2\dots \to 0$ which satisfies 
$\{(x_{h_k},\xi_{h_k})\}_{k=1}^{\infty}\subseteq B(r_j)$ and 
$ |a(x_{h_k},\xi_{h_k};h_k)| \ge h_k^{\alpha_j +\varepsilon}$.
We take $0\le \psi\in C^{\infty}_c(\mathbb{R}^n)$ such that $\psi(0)=1$. 
Then Corollary 2 implies  
\begin{align*}
&(Op_{1,h}(a) \phi _{x_{h_k},\xi_{h_k} ,h}, \psi_{x_{h_k}, \xi_{h_k}, h})_{L^2}\\
&=(2\pi)^{-\frac{n}{2}}h_k^{-n}\iint \psi\Bigl(\frac{x-x_{h_k}}{h_k^{\frac{1}{2}}}\Bigr)a(x,\xi;h_k)
 e^{i\langle x-x_{h_k}, \xi-\xi_{h_k}  \rangle /h_k} 
 \hat{\phi} \Bigl(\frac{\xi-\xi_{h_k}}{h_k^{\frac{1}{2}}}\Bigr) d\xi dx \\
&=(2\pi )^{-\frac{n}{2}}\iint a(x_{h_k}+h_k^{\frac{1}{2}} x, \xi_{h_k}+h_k^{\frac{1}{2}} \xi;h_k)
\psi(x)\hat{\phi}(\xi)e^{i\langle x, \xi \rangle}dx d\xi\\
&=\mathcal{O}(h_k^{\infty}).
\end{align*}
To estimate this integral from below, we recall the following inequality.
\begin{lemma}
If a function $f$ takes its values in the sector 
$S_{\theta_0, \theta}=\{z \in \mathbb{C}||\arg z-\theta_0|\le \theta\}$ with $0\le \theta<\frac{\pi}{2}$, then 
 \[\left|\int f dx \right| \ge \cos \theta \int |f|dx.\]
 \end{lemma}
 \begin{proof}
 By considering $e^{-i\theta_0}f$, we may assume $\theta_0=0$. Then we have $\mathrm{Re}f \ge \cos \theta |f|$ and thus 
\[ \left|\int f dx \right|\ge \mathrm{Re}\int f dx =\int \mathrm{Re}f dx\ge  \cos \theta\int  |f|dx.\]
\end{proof}
Lemma 1 implies that for $(x, \xi)\in \mathrm{supp}(\psi(x)\hat{\phi}(\xi))$
\begin{align*}
&|a(x_{h_k}+h_k^{\frac{1}{2}} x, \xi_{h_k}+h_k^{\frac{1}{2}} \xi;h_k)-a(x_{h_k}, \xi_{h_k};h_k)|\\
&=|h_k^{\frac{1}{2}}\int_0^1 ((x\cdot \partial_x+\xi \cdot \partial_{\xi})a)(x_{h_k}+th_k^{\frac{1}{2}} x, \xi_{h_k}+th_k^{\frac{1}{2}} \xi;h_k)dt|\\
&\lesssim h_k^{\frac{1}{2}}h_k^{\alpha_j -\varepsilon -\delta}
\end{align*}
since $\mathrm{supp}(\psi(x)\hat{\phi}(\xi))$ is compact and thus $(x_{h_k}+th_k^{\frac{1}{2}} x, \xi_{h_k}+th_k^{\frac{1}{2}} \xi)
\in B(r_j +s)$ for large $k$ ($s$ is that in Lemma 1). 
We recall that $|a(x_{h_k},\xi_{h_k};h_k)| \ge h_k^{\alpha_j +\varepsilon}$. 
We also recall that $\varepsilon<\frac{1}{2}(\frac{1}{2}-\delta)$ and thus 
we have $\frac{1}{2}+\alpha_j-\varepsilon-\delta>\alpha_j+\varepsilon$. 
These imply that 
\[\left|\frac{a(x_{h_k}+h_k^{\frac{1}{2}} x, \xi_{h_k}+h_k^{\frac{1}{2}} \xi;h_k)-a(x_{h_k}, \xi_{h_k};h_k)}{a(x_{h_k}, \xi_{h_k};h_k)}\right|\le \frac{1}{2}\]
for large $k$.
Thus we have
\[|a(x_{h_k}+h_k^{\frac{1}{2}} x, \xi_{h_k}+h_k^{\frac{1}{2}} \xi;h_k)|\ge \frac{1}{2} h_k^{\alpha_j +\varepsilon}\]
and
\[ \left|\arg \frac{a(x_{h_k}+h_k^{\frac{1}{2}} x, \xi_{h_k}+h_k^{\frac{1}{2}} \xi;h_k)}
{a(x_{h_k},\xi_{h_k};h_k)}\right| \le \frac{\pi}{6}\]
on the support of $\psi(x)\hat{\phi}(\xi)$. Moreover, if we take $\mathrm{supp}\psi $ sufficiently small, 
$|\arg e^{i\langle x, \xi \rangle}|\le \frac{\pi}{6}$
 on the support of $\psi(x)\hat{\phi}(\xi)$ since $\hat{\phi}$ has compact support. Thus 
 the integrand $a(x_{h_k}+h_k^{\frac{1}{2}} x, \xi_{h_k}+h_k^{\frac{1}{2}} \xi;h_k)
 \psi(x)\hat{\phi}(\xi)e^{i\langle x, \xi \rangle}$ takes its values in the sector $\{z \in \mathbb{C}||\arg z-\arg a(x_{h_k},\xi_{h_k};h_k)|\le \frac{\pi}{3}\}$ and Lemma 3 implies
\begin{align*}
 &|\iint a(x_{h_k}+h_k^{\frac{1}{2}} x, \xi_{h_k}+h_k^{\frac{1}{2}} \xi;h_k)
 \psi(x)\hat{\phi}(\xi)e^{i\langle x, \xi \rangle}dx d\xi|\\
&\ge \frac{1}{2} \iint |a(x_{h_k}+h_k^{\frac{1}{2}} x, \xi_{h_k}+h_k^{\frac{1}{2}} \xi;h_k)
\psi(x)\hat{\phi}(\xi)e^{i\langle x, \xi \rangle}|dx d\xi\\
&\ge \frac{1}{4}\iint h_k^{\alpha_j +\varepsilon}\psi(x)\hat{\phi}(\xi)
dx d\xi\\
&\gtrsim  h_k^{\alpha_j +\varepsilon}.
\end{align*}
which is a contradiction.
\end{proof}

\appendix 
\section*{Appendix}
The following lower bound on $\mathrm{WF}_h(b(x)e^{\frac{i}{h}\phi(x)})$ with no condition on $b$ or $\phi$ may be new and its proof is 
similar to that of Theorem 1.
\begin{prop}
Let $b, \phi \in C^{\infty}_c(\mathbb{R}^n)$ be $h$-independent and $\mathrm{Im} \phi \ge 0$.
Then,

(1). $\mathrm{WF}_h(b(x)e^{\frac{i}{h}\phi(x)}) \subseteq \{(x,\partial \phi(x))| x\in \mathrm{supp}b, \mathrm{Im} \phi(x)=0\}$.

(2). $\overline{\{(x,\partial \phi(x))| b(x)\not=0, \mathrm{Im} \phi(x)=0\}} \subseteq \mathrm{WF}_h(b(x)e^{\frac{i}{h}\phi(x)})$.
\end{prop} 
\begin{remark}
We note that $\partial \phi(x)$ is real if $\mathrm{Im} \phi(x)=0$, since $\mathrm{Im} \phi \ge 0$ everywhere and thus $\partial \mathrm{Im}\phi(x)=0$.
\end{remark}
\begin{corollary}
If $\phi$ is real valued, 
\[\mathrm{WF}_h(b(x)e^{\frac{i}{h}\phi(x)}) = \{(x,\partial \phi(x))| x\in \mathrm{supp}b \}.\]
\end{corollary}
\begin{remark}
Both inclusions in Proposition 3 may be strict for a complex valued $\phi$. 
For (1), we take $\phi(x)=ix^2$ and $b(x)=e^{-\frac{1}{x^2}}$ near $0$. Then $(0, 0)$ belongs to the right hand side of (1). 
On the other hand, $b(x)e^{\frac{i}{h}\phi(x)}=\mathcal{O}_{L^2}(h^{\infty})$ near $0$ and thus 
$(0, 0)\not\in \mathrm{WF}_h(b(x)e^{\frac{i}{h}\phi(x)})$ since $e^{-\frac{1}{x^2}}e^{-\frac{x^2}{h}}\le e^{-2\sqrt{\frac{1}{x^2} \cdot \frac{x^2}{h}}}
=e^{-2\sqrt{\frac{1}{h}}}$. 

For (2), we take $\phi(x)=ix^2$ and $b(x)=x^2$ near $0$. Then $(0, 0)$ does not belong to the left hand side of (2). 
On the other hand, corollary 2 implies $(0, 0)\in \mathrm{WF}_h(b(x)e^{\frac{i}{h}\phi(x)})$ 
since $(x^2 e^{-\frac{x^2}{h}}, \psi_{0, 0, h})_{L^2}=h^{\frac{n}{4}+1}\int x^2 e^{-x^2}\psi(x)dx\not =\mathcal{O}(h^{\infty})$ 
for $0\le\psi\in C_c^{\infty}(\mathbb{R}^n)$.
\end{remark}
\begin{proof}[Proof of Proposition 3]
We set $u(x)=b(x)e^{\frac{i}{h}\phi(x)}$.

(1). The upper bound is easy. Take any $(x_0, \xi_0)$ from the complement of the right hand side of (1).
If $x_0\not\in \mathrm{supp}b$, or $\mathrm{Im} \phi(x_0)\not=0$, 
then $(x_0,\xi_0)\not\in\mathrm{WF}_h(u)$ since $u$ is 
$\mathcal{O}(h^{\infty})$ near $x_0$. If $\xi_0\not=\partial \phi(x_0)$, we see that $(x_0,\xi_0)\not\in\mathrm{WF}_h(u)$ 
using the differential operator $L=\frac{hD_x}{\xi-\partial\phi(x)}$, which is well defined near $(x_0, \xi_0)$, and the integration by parts argument.

(2). Note that $\mathrm{WF}_h(u)$ is closed. Suppose that $b(x_0)\not=0$ 
and $\mathrm{Im}\phi(x_0)=0$. We now show $(x_0,\partial\phi(x_0))\in \mathrm{WF}_h(u)$.
We may assume that $x_0=0$ , $\phi(x_0)=0$ and $\partial \phi(x_0)=0$. Then Taylor's theorem implies that 
\[0\le \mathrm{Im}\phi(x)\le C|x|^2, \mspace{6mu} |\mathrm{Re}\phi(x)|\le C|x|^2 \]
near $0$. We take $\psi \in C^{\infty}_c(\mathbb{R}^n)$ such that
$\psi \ge 0$ and $\psi(0)=1$. If we take $\mathrm{supp}\psi$ sufficiently close to $0$, 
\[|\mathrm{Re}\phi(x)|\le C|x|^2 \le \frac{\pi}{6},
\mspace{6mu} |\arg b(x)/b(0)|\le \frac{\pi}{6}\]
on the support of $\psi$. If $(0,0)\not\in \mathrm{WF}_h(u)$, Corollary 2 implies that 
\begin{align*}
(u, \psi_{0,0,h})_{L^2}&=
h^{-\frac{n}{4}}\int \psi(h^{-\frac{1}{2}}x)b(x)e^{i\phi(x)/h} dx\\
&=h^{\frac{n}{4}}\int \psi(x)b(h^{\frac{1}{2}}x)e^{i\phi(h^{\frac{1}{2}}x)/h} dx
=\mathcal{O}(h^{\infty}).
\end{align*}
Since $|\arg b(h^{\frac{1}{2}}x)/b(0)|\le \frac{\pi}{6}$ and
\[|\mathrm{Re}\phi(h^{\frac{1}{2}}x)/h|\le C|h^{\frac{1}{2}}x|^2/h =C|x|^2 \le \frac{\pi}{6}\]
on the support of $\psi$, the integrand $\psi(x)b(h^{\frac{1}{2}}x)e^{i\phi(h^{\frac{1}{2}}x)/h}$ takes its values in the sector 
$\{z \in \mathbb{C}| |\arg z-\arg b(0)|\le \frac{\pi}{3}\}$ and Lemma 3 implies  
\[|\int \psi(x)b(h^{\frac{1}{2}}x)e^{i\phi(h^{\frac{1}{2}}x)/h} dx|
\ge \frac{1}{2}\int \psi(x)|b(h^{\frac{1}{2}}x)|e^{-\mathrm{Im}\phi(h^{\frac{1}{2}}x)/h} dx.\]
Since we also have $0\le \mathrm{Im}\phi(h^{\frac{1}{2}}x)/h\lesssim 1$ as in the case of the real part,
we obtain by Fatou's lemma, 
 \[\liminf_{h\to 0} \int \psi(x)|b(h^{\frac{1}{2}}x)|e^{-\mathrm{Im}\phi(h^{\frac{1}{2}}x)/h} dx
\gtrsim \int \psi(x)|b(0)|dx>0\]
which is a contradiction.
\end{proof}

\section*{Acknowledgement}
The author is grateful to his adviser Shu Nakamura and Professor Maciej Zworski for comments and encouragement. 
He is also grateful to the anonymous referee for suggestions to improve the manuscript.
The author is under the support of the FMSP program at the Graduate School of Mathematical Sciences, University of Tokyo.

Graduate School of Mathematical Sciences, University of Tokyo,
 3-8-1, Komaba, Meguro-ku, Tokyo 153-8914, Japan
 
 kameoka@ms.u-tokyo.ac.jp

\end{document}